\documentclass[12pt]{amsart}
\usepackage{
	amssymb,
	amsfonts,
	bm,
	hyperref,
	mathtools,
	xcolor,
	tikz-cd,
}

\usetikzlibrary{arrows.meta, positioning}
\tikzset{
	none/.style={line width=2mm},
	vertex/.style={circle, fill, inner sep=1pt},
	edge/.style={draw}
}
\pgfdeclarelayer{nodelayer}
\pgfdeclarelayer{edgelayer}
\pgfsetlayers{edgelayer,nodelayer,main}

\usetikzlibrary{decorations.markings}
\tikzset{
	mid arrow/.style={
		postaction={decorate,
			decoration={
				markings,
				mark=at position 0.6 with {\arrow{stealth}}
			}
		}
	}
}

\newtheorem{thm}{Theorem}[section]

\newtheorem{prop}[thm]{Proposition}
\newtheorem{ex}[thm]{Example}
\newtheorem{cor}[thm]{Corollary}
\newtheorem{rem}[thm]{Remark}

\DeclareMathOperator{\tr}{tr}

\newcommand{\Dfn}[1]{\emph{\color{blue}#1}}
\newcommand{\oeis}[1]{\href{https://oeis.org/#1}{#1}}
\newcommand*{\kk}{\Bbbk}
\newcommand\Spec{\mathrm{Spec}}
\newcommand\Aut{\mathrm{Aut}}
\newcommand\SL{\mathrm{SL}}
\newcommand\PSO{\mathrm{PSO}}
\newcommand\Spin{\mathrm{Spin}}
\newcommand\OSp{\mathrm{OSp}}
\newcommand\SU{\mathrm{SU}}
\newcommand\bA{\mathbb{A}}
\newcommand\bC{\mathbb{C}}

\newcommand\bO{\mathbb{O}}
\newcommand\bN{\mathbb{N}}
\newcommand\bQ{\mathbb{Q}}
\newcommand\bZ{\mathbb{Z}}
\newcommand\bR{\mathbb{R}}

\newcommand\fg{\mathfrak{g}}
\newcommand\fsl{\mathfrak{sl}}
\newcommand\fgl{\mathfrak{gl}}
\newcommand\fso{\mathfrak{so}}
\newcommand\fS{\mathfrak{S}}
\newcommand\nn{\mathbf{n}}
\newcommand\der{\mathbf{der}}
\newcommand\str{\mathbf{str}}
\newcommand\cT{\mathcal{T}}
\newcommand\cB{\mathcal{B}}

\begin{document}
\title{Cubic Jordan algebras are not a series}
\author{Bruce Westbury}
\date{\today}
\begin{abstract}
The idea of the exceptional series is that the exceptional simple Lie algebras should form a
series. Since all four simple Lie algebras in the fourth row of the Freudenthal magic square
are exceptional it is natural to ask if the remaining rows form a series. A stronger version
of this question is that, for the first two rows (corresponding to $\bR$ and $\bC$), there is a category defined by a presentation which is a reasonable candidate for the series. Our main results
show that neither of these candidates is a series but each consists of a finite set of points.
In each case the series is defined by a parameter and we show that the relations imply
that this parameter satisfies a polynomial. These two results were obtained by a computer calculation.
Our calculation is supported by the website, \url{www.brucewestbury.uk}, for inspection, and the calculations are certified by Lean 4.
\end{abstract}
\maketitle
\tableofcontents

\tikzset{bwwpicture/.style={rotate=0,baseline={([yshift=-\the\dimexpr\fontdimen22\textfont2\relax]current  bounding  box.center)}, red, ultra thick,scale=.9,line join=round}}
\tikzset{bwwcircle/.style={bwwpicture,execute at begin picture={\path (0,0) circle [radius=1.2];  \fill[bgdfill] circle [radius=1.0];},
		execute at end picture={\draw[bgddraw] circle [radius=1.0];}}}
\tikzset{bwwrect/.style 2 args={bwwpicture,execute at begin picture={\fill[bgdfill] #1 coordinate (a) rectangle #2 coordinate (b); \path (barycentric cs:a=1.11,b=-.11) rectangle (barycentric cs:a=-.11,b=1.11); \clip (a) rectangle (b);},
		execute at end picture={\draw[bgddraw] #1 rectangle #2;}},bwwrect/.default={(0,0)}{(1.5,1.5)}}

\tikzstyle{bgddraw}=[blue, thin]
\tikzstyle{bgdfill}=[fill=blue!20]
\tikzstyle{wipe}=[double distance=1.6pt,double=red,blue!20,line width=3pt]

\def\trivalent{
	\begin{tikzpicture}[bwwcircle]
		\draw  (0,0) -- (270:1);
		\draw  (0,0) -- (30:1);
		\draw  (0,0) -- (150:1);
	\end{tikzpicture}
}

\def\circle{
	\begin{tikzpicture}[bwwcircle]
		\draw  (0,0) circle [radius=0.5];		
	\end{tikzpicture}
}

\def\empty{
	\begin{tikzpicture}[bwwcircle]
	\end{tikzpicture}
}

\def\tadpole{
	\begin{tikzpicture}[bwwcircle]
		\draw  (0,-0.2) -- (270:1);
		\draw  (0,-0.2) to[out=30,in=270] +(45:0.5) arc(0:180:0.3535) to[out=270,in=150] (0,-0.2);
	\end{tikzpicture}
}

\def\bigon{
	\begin{tikzpicture}[bwwcircle]
		\draw  (0,0.5) to [out=210, in=90] (-0.4,0) to [out=270, in=150] (0,-0.5);
		\draw  (0,0.5) to [out=330, in=90] (0.4,0) to [out=270, in=30](0,-0.5);
		\draw  (0,0.5) -- (0,1);
		\draw  (0,-0.5) -- (0,-1);
	\end{tikzpicture}
}

\def\linecirc{
	\begin{tikzpicture}[bwwcircle]
		\draw  (0,1) -- (0,-1);
	\end{tikzpicture}
}

\def\trianglecirc{
	\begin{tikzpicture}[bwwcircle]
		\draw  (90:0.5) -- (90:1);
		\draw  (210:0.5) -- (210:1);
		\draw  (330:0.5) -- (330:1);
		\draw  (90:0.5) to [out=210, in=90] (210:0.5);
		\draw  (210:0.5) to [out=330, in=210] (330:0.5);
		\draw  (330:0.5) to [out=90, in=330] (90:0.5);
	\end{tikzpicture}
}

\def\Xcirc{
	\begin{tikzpicture}[bwwcircle]
		\draw  (45:1) -- (225:1);
		\draw  (135:1) -- (315:1);
	\end{tikzpicture}
}

\def\Icirc{
	\begin{tikzpicture}[bwwcircle]
		\draw  (45:1) to [out=225,in=135] (315:1);
		\draw  (135:1) to [out=315,in=45] (225:1);
	\end{tikzpicture}
}

\def\Ucirc{
	\begin{tikzpicture}[bwwcircle]
		\draw  (45:1) to [out=225,in=315] (135:1);
		\draw  (315:1) to [out=135,in=45] (225:1);
	\end{tikzpicture}
}

\def\Kcirc{
	\begin{tikzpicture}[bwwcircle]
		\draw  (45:1) to [out=225,in=30] (0,0.3);
		\draw  (135:1) to [out=315,in=150] (0,0.3);
		\draw  (225:1) to [out=45,in=210] (0,-0.3);
		\draw  (315:1) to [out=135,in=330] (0,-0.3);
		\draw  (0,0.3) -- (0,-0.3);
	\end{tikzpicture}
}

\def\Hcirc{
	\begin{tikzpicture}[bwwcircle]
		\draw  (45:1) to [out=225,in=60] (0.3,0);
		\draw  (135:1) to [out=315,in=120] (-0.3,0);
		\draw  (225:1) to [out=45,in=240] (-0.3,0);
		\draw  (315:1) to [out=135,in=300] (0.3,0);
		\draw  (0.3,0) -- (-0.3,0);
	\end{tikzpicture}
}

\def\squarecirc{
	\begin{tikzpicture}[bwwcircle]
		\draw  (45:0.5) -- (45:1);
		\draw  (135:0.5) -- (135:1);
		\draw  (225:0.5) -- (225:1);
		\draw  (315:0.5) -- (315:1);
		\draw  (45:0.5) to [out=165,in=15] (135:0.5);
		\draw  (135:0.5) to [out=255,in=105] (225:0.5);
		\draw  (225:0.5) to[out=345,in=195] (315:0.5);
		\draw  (315:0.5) to [out=75,in=285] (45:0.5);
	\end{tikzpicture}
}

\def\twosquare{
	\begin{tikzpicture}[bwwcircle]
		\def\x{0.3}
		\coordinate (A) at (0,\x) {};
		\coordinate (B) at (0,-\x) {};
		\draw  (45:0.5) -- (45:1);
		\draw  (135:0.5) -- (135:1);
		\draw  (225:0.5) -- (225:1);
		\draw  (315:0.5) -- (315:1);
		\draw  (45:0.5) -- (A) -- (135:0.5);
		\draw  (135:0.5) -- (225:0.5);
		\draw  (225:0.5) -- (B) -- (315:0.5);
		\draw  (315:0.5) -- (45:0.5);
		\draw (A) -- (B);
	\end{tikzpicture}
}

\def\Xpos{
	\begin{tikzpicture}[bwwcircle]
		\posdot{0,0}; %Not defined
		\draw (45:1) -- (225:1);
		\draw	(135:1) -- (315:1);
	\end{tikzpicture}
}

\def\Acirc{
	\begin{tikzpicture}[bwwcircle]
		\draw (45:1) to [out=225,in=105] (225:0.5);
		\draw (135:1) to [out=315,in=75] (315:0.5);
		\draw (225:0.5) to [out=345,in=195] (315:0.5);
		\draw (225:0.5) -- (225:1);
		\draw (315:0.5) -- (315:1);
	\end{tikzpicture}
}

\def\overlaptwotwo{
	\begin{tikzpicture}[bwwcircle]
		\coordinate (A) at (0,0.5);
		\coordinate (B) at (0,-0.5);
		\draw (A) -- (B);
		\draw (A) to[out=150,in=90] (-0.5,0) to [out=270,in=210] (B);
		\draw (A) to[out=30,in=90] (0.5,0) to [out=270,in=330] (B);
	\end{tikzpicture}
}

\def\overlaptwothree{
	\begin{tikzpicture}[bwwcircle]
		\coordinate (A) at (120:0.5);
		\coordinate (B) at (240:0.5);
		\coordinate (C) at (0:0.5);
		\draw (A) to[out=90,in=120] (C);
		\draw (B) to[out=270,in=240] (C);
		\draw (C) -- (0:1);
		\draw (A) to[out=330,in=30] (B);
		\draw (A) to[out=210,in=120] (B);
	\end{tikzpicture}
}

\def\overlaptwofour{
	\begin{tikzpicture}[bwwcircle]
		\coordinate (A) at (45:0.5);
		\coordinate (B) at (135:0.5);
		\coordinate (C) at (225:0.5);
		\coordinate (D) at (315:0.5);
		\draw (A) -- (45:1);
		\draw (D) -- (315:1);
		\draw (A) -- (B);
		\draw (B) to[out=330,in=30] (C);
		\draw (B) to[out=210,in=120] (C);
		\draw (C) -- (D) -- (A);
	\end{tikzpicture}
}

\def\overlapthreethree{
	\begin{tikzpicture}[bwwcircle]
		\coordinate (A) at (0:0.5);
		\coordinate (B) at (90:0.5);
		\coordinate (C) at (180:0.5);
		\coordinate (D) at (270:0.5);
		\draw (A) -- (B) -- (C) -- (D) -- cycle;
		\draw (A) -- (1,0) (B) -- (D) (C) -- (-1,0);
	\end{tikzpicture}
}

\def\overlapthreefour{
	\begin{tikzpicture}[bwwcircle]
		\coordinate (A) at (0:0.5);
		\coordinate (B) at (72:0.5);
		\coordinate (C) at (144:0.5);
		\coordinate (D) at (216:0.5);
		\coordinate (E) at (288:0.5);
		\draw (A) -- (B) -- (C) -- (D) -- (E) -- cycle;
		\draw (A) -- (1,0) (B) -- (E) (C) -- (120:1) (D) -- (240:1);
	\end{tikzpicture}
}

\def\overlapfourfour{
	\begin{tikzpicture}[bwwcircle]
		\coordinate (A) at (30:0.5);
		\coordinate (B) at (90:0.5);
		\coordinate (C) at (150:0.5);
		\coordinate (D) at (210:0.5);
		\coordinate (E) at (270:0.5);
		\coordinate (F) at (330:0.5);
		\draw (A) -- (B) -- (C) -- (D) -- (E) -- (F) -- cycle;
		\draw (B) -- (E) (A) -- (45:1) (C) -- (135:1) (D) -- (225:1) (F) -- (315:1);
	\end{tikzpicture}
}

\def\overlaptwotwoA{
	\begin{tikzpicture}[bwwcircle]
		\coordinate (A) at (0,0.5);
		\coordinate (B) at (0,-0.5);
		\draw (A) -- (B) (0,0) -- (1,0);
		\draw (A) to[out=150,in=90] (-0.5,0) to [out=270,in=210] (B);
		\draw (A) to[out=30,in=90] (0.5,0) to [out=270,in=330] (B);
	\end{tikzpicture}
}

\def\overlaptwothreeA{
	\begin{tikzpicture}[bwwcircle]
		\coordinate (A) at (120:0.5);
		\coordinate (B) at (240:0.5);
		\coordinate (C) at (0:0.5);
		\draw (A) to[out=90,in=120] (C);
		\draw (B) to[out=270,in=240] (C);
		\draw (C) -- (0:1);
		\draw (A) to[out=210,in=120] (B);
		\draw (A) to[out=210,in=90] (0,0) to[out=270,in=120] (B);
		\draw (0,0) -- (-1,0);
	\end{tikzpicture}
}

\def\overlapthreethreeA{
	\begin{tikzpicture}[bwwcircle]
		\coordinate (A) at (0:0.5);
		\coordinate (B) at (90:0.5);
		\coordinate (C) at (180:0.5);
		\coordinate (D) at (270:0.5);
		\draw (A) -- (B) -- (C) -- (D) -- cycle;
		\draw (A) -- (1,0) (B) -- (D) (C) -- (-1,0) (0,0) -- (45:1);
	\end{tikzpicture}
}

\def\overlaptwotwoB{
	\begin{tikzpicture}[bwwcircle]
		\coordinate (A) at (0,0.5);
		\coordinate (B) at (0,-0.5);
		\draw (A) -- (B);
		\draw (A) to[out=150,in=90] (-0.5,0) to [out=270,in=210] (B);
		\draw (A) to[out=30,in=90] (0.5,0) to [out=270,in=330] (B);
		\draw (0,0.2) -- (45:1) (0,-0.2) -- (215:1);
	\end{tikzpicture}
}

\def\curl{
	\begin{tikzpicture}[bwwcircle]
		\draw (0.25,0.45) arc (0:180:0.25) to[out=270,in=135] (0,0) -- (315:1);
		\draw (225:1) -- (0,0) to[out=225,in=270] (0.25,0.45);
	\end{tikzpicture}
}

\def\twist{
	\begin{tikzpicture}[bwwcircle]
		\draw  (30:1) to [out=210,in=45] (0,0.2) to [out=225, in=180] (0,-0.5);
		\draw (150:1) to [out=330,in=135] (0,0.2) to [out=-45, in=0] (0,-0.5);
		\draw  (0,-0.5) -- (0,-1);
	\end{tikzpicture}
}

%------------- Directed --------------------

\def\source{
	\begin{tikzpicture}[bwwcircle]
		\draw[mid arrow] (0,0) -- (270:1);
		\draw[mid arrow] (0,0) -- (30:1);
		\draw[mid arrow] (0,0) -- (150:1);
	\end{tikzpicture}
}

\def\sink{
	\begin{tikzpicture}[bwwcircle]
		\draw[mid arrow] (270:1) -- (0,0);
		\draw[mid arrow] (30:1) -- (0,0);
		\draw[mid arrow] (150:1) -- (0,0);
	\end{tikzpicture}
}

\def\circledir{
	\begin{tikzpicture}[bwwcircle]
		\draw[mid arrow]  (0,0) circle [radius=0.5];		
	\end{tikzpicture}
}

\def\bigondir{
	\begin{tikzpicture}[bwwcircle]
	\draw[mid arrow]  (0,0.5) to [out=210, in=90] (-0.4,0) to [out=270, in=150] (0,-0.5);
	\draw[mid arrow]  (0,0.5) to [out=330, in=90] (0.4,0) to [out=270, in=30](0,-0.5);
	\draw[mid arrow]  (0,0.5) -- (0,1);
	\draw[mid arrow]  (0,-1) -- (0,-0.5);
\end{tikzpicture}
}

\def\linedir{
\begin{tikzpicture}[bwwcircle]
	\draw[mid arrow]  (0,-1) -- (0,1);
\end{tikzpicture}
}

\def\squaredir{
	\begin{tikzpicture}[bwwcircle]
		\draw[mid arrow]  (45:1) -- (45:0.5);
		\draw[mid arrow]  (135:0.5) -- (135:1);
		\draw[mid arrow]  (225:1) -- (225:0.5);
		\draw[mid arrow]  (315:0.5) -- (315:1);
		\draw[mid arrow]  (135:0.5) -- (45:0.5); %to [out=165,in=15] (135:0.5);
		\draw[mid arrow]  (135:0.5) to [out=255,in=105] (225:0.5);
		\draw[mid arrow]  (315:0.5) -- (225:0.5); %to[out=345,in=195] (315:0.5);
		\draw[mid arrow]  (315:0.5) to [out=75,in=285] (45:0.5);
	\end{tikzpicture}
}

\def\Adir{
	\begin{tikzpicture}[bwwcircle]
		\draw[mid arrow] (45:1) to [out=225,in=105] (225:0.5);
		\draw[mid arrow] (315:0.5) to [out=60,in=315] (135:1); %(135:1) to [out=315,in=75] (315:0.5);
		\draw[mid arrow] (315:0.5) -- (225:0.5); %(225:0.5) to [out=345,in=195] (315:0.5);
		\draw[mid arrow] (225:1) -- (225:0.5);
		\draw[mid arrow] (315:0.5) -- (315:1);
	\end{tikzpicture}
}

\def\Idir{
	\begin{tikzpicture}[bwwcircle]
		\draw[mid arrow]  (45:1) to [out=225,in=135] (315:1);
		\draw[mid arrow]  (225:1) to [out=45,in=315] (135:1);
	\end{tikzpicture}
}

\def\Udir{
	\begin{tikzpicture}[bwwcircle]
		\draw[mid arrow]  (45:1) to [out=225,in=315] (135:1);
		\draw[mid arrow]  (225:1) to [out=45,in=135] (315:1);
	\end{tikzpicture}
}

\def\Wa{
	\begin{tikzpicture}[bwwcircle]
		\coordinate (b0) at (90:1);
		\coordinate (b1) at (215:1);
		\coordinate (b2) at (255:1);
		\coordinate (b3) at (285:1);
		\coordinate (b4) at (325:1);
		\coordinate (v0) at (90:0.5);
		\coordinate (v1) at (-0.4,0);
		\coordinate (v2) at (0.4,0);
		\draw[mid arrow] (b0) -- (v0);
		\draw[mid arrow] (v0) -- (v1);
		\draw[mid arrow] (v0) -- (v2);
		\draw[mid arrow] (v1) -- (b1);
		\draw[mid arrow] (v1) -- (b2);
		\draw[mid arrow] (v2) -- (b3);
		\draw[mid arrow] (v2) -- (b4);
	\end{tikzpicture}
}

\def\Wb{
	\begin{tikzpicture}[bwwcircle]
		\coordinate (b0) at (90:1);
		\coordinate (b1) at (215:1);
		\coordinate (b2) at (255:1);
		\coordinate (b3) at (285:1);
		\coordinate (b4) at (325:1);
		\coordinate (v0) at (90:0.5);
		\coordinate (v1) at (-0.4,0);
		\coordinate (v2) at (0.4,0);
		\draw[mid arrow] (b0) -- (v0);
		\draw[mid arrow] (v0) -- (v1);
		\draw[mid arrow] (v0) -- (v2);
		\draw[mid arrow] (v1) -- (b1);
		\draw[mid arrow] (v1) -- (b3);
		\draw[mid arrow] (v2) -- (b2);
		\draw[mid arrow] (v2) -- (b4);
	\end{tikzpicture}
}

\def\Wc{
	\begin{tikzpicture}[bwwcircle]
		\coordinate (b0) at (90:1);
		\coordinate (b1) at (215:1);
		\coordinate (b2) at (255:1);
		\coordinate (b3) at (285:1);
		\coordinate (b4) at (325:1);
		\coordinate (v0) at (90:0.5);
		\coordinate (v1) at (-0.4,0);
		\coordinate (v2) at (0.4,0);
		\draw[mid arrow] (b0) -- (v0);
		\draw[mid arrow] (v0) -- (v1);
		\draw[mid arrow] (v0) -- (v2);
		\draw[mid arrow] (v1) -- (b1);
		\draw[mid arrow] (v1) -- (b4);
		\draw[mid arrow] (v2) -- (b2);
		\draw[mid arrow] (v2) -- (b3);
	\end{tikzpicture}
}

\def\Vu{
	\begin{tikzpicture}[bwwcircle]
		\coordinate (b0) at (90:1);
		\coordinate (b1) at (215:1);
		\coordinate (b2) at (255:1);
		\coordinate (b3) at (285:1);
		\coordinate (b4) at (325:1);
		\coordinate (v0) at (90:0.5);
		\draw[mid arrow] (b0) -- (b1);
		\draw[mid arrow] (v0) -- (b2);
		\draw[mid arrow] (v0) -- (b3);
		\draw[mid arrow] (v0) -- (b4);
	\end{tikzpicture}
}

\def\Vv{
	\begin{tikzpicture}[bwwcircle]
		\coordinate (b0) at (90:1);
		\coordinate (b1) at (215:1);
		\coordinate (b2) at (255:1);
		\coordinate (b3) at (285:1);
		\coordinate (b4) at (325:1);
		\coordinate (v0) at (90:0.5);
		\draw[mid arrow] (b0) -- (b2);
		\draw[mid arrow] (v0) -- (b1);
		\draw[mid arrow] (v0) -- (b3);
		\draw[mid arrow] (v0) -- (b4);
	\end{tikzpicture}
}

\def\Vw{
	\begin{tikzpicture}[bwwcircle]
		\coordinate (b0) at (90:1);
		\coordinate (b1) at (215:1);
		\coordinate (b2) at (255:1);
		\coordinate (b3) at (285:1);
		\coordinate (b4) at (325:1);
		\coordinate (v0) at (90:0.5);
		\draw[mid arrow] (b0) -- (b3);
		\draw[mid arrow] (v0) -- (b1);
		\draw[mid arrow] (v0) -- (b2);
		\draw[mid arrow] (v0) -- (b4);
	\end{tikzpicture}
}

\def\Vx{
	\begin{tikzpicture}[bwwcircle]
		\coordinate (b0) at (90:1);
		\coordinate (b1) at (215:1);
		\coordinate (b2) at (255:1);
		\coordinate (b3) at (285:1);
		\coordinate (b4) at (325:1);
		\coordinate (v0) at (90:0.5);
		\draw[mid arrow] (b0) -- (b4);
		\draw[mid arrow] (v0) -- (b1);
		\draw[mid arrow] (v0) -- (b2);
		\draw[mid arrow] (v0) -- (b3);
	\end{tikzpicture}
}

\def\overlaptwotwodir{
	\begin{tikzpicture}[bwwcircle]
		\coordinate (A) at (0,0.5);
		\coordinate (B) at (0,-0.5);
		\draw[mid arrow] (A) -- (B);
		\draw[mid arrow] (A) to[out=150,in=90] (-0.5,0) to [out=270,in=210] (B);
		\draw[mid arrow] (A) to[out=30,in=90] (0.5,0) to [out=270,in=330] (B);
	\end{tikzpicture}
}

\def\overlaptwofourdir{
	\begin{tikzpicture}[bwwcircle]
		\coordinate (A) at (45:0.5);
		\coordinate (B) at (135:0.5);
		\coordinate (C) at (225:0.5);
		\coordinate (D) at (315:0.5);
		\draw[mid arrow] (45:1) -- (A);
		\draw[mid arrow] (D) -- (315:1);
		\draw[mid arrow] (B) -- (A);
		\draw[mid arrow] (B) to[out=330,in=30] (C);
		\draw[mid arrow] (B) to[out=210,in=120] (C);
		\draw[mid arrow] (D) -- (C);
		\draw[mid arrow] (D) -- (A);
	\end{tikzpicture}
}

\def\overlapfourfourdir{
	\begin{tikzpicture}[bwwcircle]
		\coordinate (A) at (30:0.5);
		\coordinate (B) at (90:0.5);
		\coordinate (C) at (150:0.5);
		\coordinate (D) at (210:0.5);
		\coordinate (E) at (270:0.5);
		\coordinate (F) at (330:0.5);
		\draw[mid arrow] (A) -- (B);
		\draw[mid arrow] (C) -- (B);
		\draw[mid arrow] (C) -- (D);
		\draw[mid arrow] (E) -- (D);
		\draw[mid arrow] (E) -- (F);
		\draw[mid arrow] (A) -- (F);
		\draw[mid arrow] (E) -- (B);
		\draw[mid arrow] (A) -- (45:1);
		\draw[mid arrow] (C) -- (135:1);
		\draw[mid arrow] (225:1) -- (D);
		\draw[mid arrow] (315:1) -- (F);
	\end{tikzpicture}
}

\def\overlapthreethreeAdir{
	\begin{tikzpicture}[bwwcircle]
		\coordinate (A) at (0:0.5);
		\coordinate (B) at (90:0.5);
		\coordinate (C) at (180:0.5);
		\coordinate (D) at (270:0.5);
		\draw[mid arrow] (A) -- (B);
		\draw[mid arrow] (C) -- (B);
		\draw[mid arrow] (C) -- (D);
		\draw[mid arrow] (A) -- (D);
		\draw[mid arrow] (A) -- (1,0);
		\draw[mid arrow] (0,0) -- (B);
		\draw[mid arrow] (0,0) -- (D);
		\draw[mid arrow] (C) -- (-1,0);
		\draw[mid arrow] (0,0) -- (45:1);
	\end{tikzpicture}
}

\def\overlaptwotwoBdir{
	\begin{tikzpicture}[bwwcircle]
		\coordinate (A) at (0,0.5);
		\coordinate (B) at (0,-0.5);
		\draw[mid arrow] (B) -- (A);
		\draw[mid arrow] (A) to[out=150,in=90] (-0.5,0) to [out=270,in=210] (B);
		\draw[mid arrow] (A) to[out=30,in=90] (0.5,0) to [out=270,in=330] (B);
		\draw[mid arrow] (45:1) -- (0,0.2);
		\draw[mid arrow] (0,-0.2) -- (215:1);
	\end{tikzpicture}
}

%------------- Redundant --------------------

\def\leftover{
	\begin{tikzpicture}[bwwcircle,rotate=90]
		\draw (45:1) to [out=225,in=105] (225:0.5);
		\draw [wipe] (135:1) to [out=315,in=75] (315:0.5);
		\draw (225:0.5) to [out=345,in=195] (315:0.5);
		\draw (225:0.5) -- (225:1);
		\draw (315:0.5) -- (315:1);
	\end{tikzpicture}
}

\def\botover{
	\begin{tikzpicture}[bwwcircle,rotate=180]
		\draw (45:1) to [out=225,in=105] (225:0.5);
		\draw [wipe] (135:1) to [out=315,in=75] (315:0.5);
		\draw (225:0.5) to [out=345,in=195] (315:0.5);
		\draw (225:0.5) -- (225:1);
		\draw (315:0.5) -- (315:1);
	\end{tikzpicture}
}

\def\rightover{
	\begin{tikzpicture}[bwwcircle,rotate=270]
		\draw  (45:1) to [out=225,in=105] (225:0.5);
		\draw [wipe] (135:1) to [out=315,in=75] (315:0.5);
		\draw  (225:0.5) to [out=345,in=195] (315:0.5);
		\draw  (225:0.5) -- (225:1);
		\draw  (315:0.5) -- (315:1);
	\end{tikzpicture}
}

\def\topoverbar{
	\begin{tikzpicture}[bwwcircle,rotate=180,xscale=-1]
		\draw (45:1) to [out=225,in=105] (225:0.5);
		\draw [wipe] (135:1) to [out=315,in=75] (315:0.5);
		\draw (225:0.5) to [out=345,in=195] (315:0.5);
		\draw (225:0.5) -- (225:1);
		\draw (315:0.5) -- (315:1);
	\end{tikzpicture}
}

\def\pentagon{
	\begin{tikzpicture}[bwwcircle]
		\def\x{0.5}
		\draw (0:1) -- (0:\x) (72:1) -- (72:\x) (144:1) -- (144:\x) (216:1) -- (216:\x) (288:1) -- (288:\x);
		\draw  (0:\x) -- (72:\x) -- (144:\x) -- (216:\x) -- (288:\x) -- cycle; 
	\end{tikzpicture}
}

\def\tree#1{
	\begin{tikzpicture}[bwwcircle,rotate={#1}]
		\def\x{0.25}
		\coordinate (A) at (-\x+0.1,\x) {};
		\coordinate (B) at (\x+0.1,0) {};
		\coordinate (C) at (-\x+0.1,-\x) {};
		\draw (A) -- (B) -- (C);
		\draw (0:1) -- (B) (72:1) -- (A) (144:1) -- (A) (216:1) -- (C) (288:1) -- (C);
	\end{tikzpicture}
}

\def\fivesq#1{
	\begin{tikzpicture}[bwwcircle,rotate={#1}]
		\def\x{0.35}
		\coordinate (A) at (-0.5,0) {};
		\coordinate (N) at (\x,\x) {};
		\coordinate (S) at (\x,-\x) {};
		\coordinate (E) at (2*\x,0) {};
		\coordinate (W) at (0,0) {};
		\draw (N) -- (E) -- (S) -- (W) -- cycle;
		\draw (A) -- (W) (A) -- (144:1) (A) -- (216:1);
		\draw (0:1) -- (E) (72:1) -- (N) (288:1) -- (S);
	\end{tikzpicture}
}

\def\twothree#1{
	\begin{tikzpicture}[bwwcircle,rotate={#1}]
		\coordinate (A) at (0:0.1) {};
		\draw (144:1) to[out=324,in=36] (216:1);
		\draw (0:1) to[out=180,in=0] (A) (72:1) to[out=252,in=120] (A) (288:1) to[out=108,in=240] (A);
	\end{tikzpicture}
}

%%% Local Variables:
%%% TeX-master: "f4_closed.tex"
%%% End:

%============= 
% Introduction
%=============
\section{Introduction}
The classification of complex simple Lie algebras is due to Killing in 1888. The classification consists of four infinite families, known as the classical Lie algebras, and five others, known as the exceptional Lie algebras.

A more recent development was the discovery of the Vogel plane, partially published in \cite{Vogel2011}.
Vogel proved (independently of Killing's classification) that the complement of the trivial representation in the symmetric square of the adjoint representation of a simple Lie algebra (other than $\fsl(2)$) has either two or three composition factors. Furthermore, for the exceptional algebras, the number of composition factors is two.
The four infinite families of classical Lie algebras are often regarded as forming a series; for example,
the ring of symmetric functions can be regarded as the character ring of $\fgl(n)$ where $n$ is a formal variable. The radical idea in Vogel's work is that the exceptional Lie algebras might also form 
a series.

In earlier work,\cite{Freudenthal1964}, Freudenthal gave a uniform matrix construction of the exceptional
Lie algebras. This construction is now known as the Freudenthal magic square. The fourth row of the magic square consists of four of the exceptional Lie algebras. Since Vogel had proposed that the fourth row might form
a series it is natural to ask about the remaining rows. % This idea was investigated in \cite{Deligne2002}.
The result announced in \cite{Thurston2004} is the negative answer that the first two rows of the magic square do not form a series. This result was obtained by a computer calculation which cannot be inspected and which has not been reproduced.
This work reconstructs the calculation from first principles using diagrammatic reduction rules. produces a transparent reduction system, and supplies independently checkable certificates for the computation.

This work is supported by the website \url{www.brucewestbury.uk}
which has the role of an Appendix and presents details of the calculations.
This work is further supported by Lean code which independently certifies the reduction steps.
The coding was done by ChatGPT and Claude.

%Warning: The $\nn$ for the derivation algebra and the $\nn$ for the structure algebra are different.

\subsection*{Acknowledgment}
The author would like to thank Dylan Thurston for valuable conversations.

\section{Cubic Jordan algebras}\label{sec:cubic}
The classical Lie algebras all have matrix constructions; so the classification of simple Lie algebras
raised the problem of giving matrix constructions of the exceptional Lie algebras.
The Lie algebra $G_2$ was shown to be the derivation algebra of the octonions by Elie Cartan in 1913. Then matrix constructions of $F_4$ and $E_6$ were given in \cite{freudenthal1953} by constructing actions on the
Jordan algebra of $3\times 3$ Hermitian matrices with entries in the octonions. This algebra is known as
the Albert algebra.

%The original papers are \cite{chevalley1950}, \cite{tits1966}, \cite{freudenthal1953}.
 
%A more geometric interpretation is given in \cite[\S4.2,\S4.4]{baez2002}.

\subsection{Construction}
An accessible account of these constructions is given in \cite{adams1982}.
Let $\bA$ be a (unital) \Dfn{composition algebra} with involution $x\mapsto \overline{x}$ and norm $N(x)\coloneqq x\overline{x}$. The condition for a composition algebra is
\begin{equation*}
	N(xy) = N(x)N(y)
\end{equation*}

An element $x\in \bA$ is \Dfn{real} if $\overline{x}=x$ and \Dfn{imaginary} if
$\overline{x}=-x$. Every $x\in \bA$ has a real part $\Re(x)\coloneqq \frac12(x+\overline{x})$
and an imaginary part $\Im(x)\coloneqq \frac12(x-\overline{x})$ with $x=\Re(x)+\Im(x)$.

Let $\bA$ be a composition algebra of dimension $m\in \{1,2,4,8\}$.
Let $H_3(\bA)$ be the space of Hermitian $3\times 3$ matrices with entries in $\bA$.
The space $H_3(\bA)$ has dimension $3m+3$.

Define the Jordan product of $X,Y\in H_3(\bA)$ by
\begin{equation*}
	X \circ Y \coloneqq \frac12 (XY+YX)
\end{equation*}
It is clear that the Jordan product is commutative and power associative; also, the unit matrix is a unit for the Jordan product.
The \Dfn{Jordan identity} is $x(x^2y) = x^2(xy)$ for all $x,y$. 
The Jordan product also satisfies the Jordan identity.
This is clear if $\bA$ is associative, and, remarkably, also holds for $\bA=\bO$, the octonions. 

The Jordan algebra $H_3(\bA)$ is a cubic Jordan algebra. An element $X\in H_3(\bA)$
can be written as
\begin{equation*}
	X=\begin{pmatrix}
		\alpha & x & \overline{z} \\
		\overline{x} & \beta & y \\
		z & \overline{y} & \gamma
	\end{pmatrix}
\end{equation*}
where the diagonal entries $\alpha,\beta,\gamma$ are scalars and the remaining entries are in $\bA$.
The \Dfn{norm}, $N$, is given explicitly by
\begin{equation*}
	N(X) = \alpha\beta\gamma - \alpha N(y) - \beta N(z) - \gamma N(x) + \Re ((xy)z)
\end{equation*}

Define a symmetric inner product using the matrix trace by
\begin{equation*}
	(X,Y) = \tr (X\circ Y)
\end{equation*}

The properties of the inner product are that it is non-degenerate and invariant where invariance
is the condition that,for $X,Y,Z\in H_3(\bA)$,
\begin{equation*}
	(X,Y\circ Z) = (X\circ Y,Z)
\end{equation*}

Let $\fg_\bA$ be the derivation algebra of $H_3(\bA)$. Then $\fg_\bA$ preserves the inner product.
Let $H'_3(\bA)\subset H_3(\bA)$ be the subspace of trace-free matrices. Then $H_3(\bA)$ decomposes
as a sum of irreducible $\fg_\bA$-modules
\begin{equation*}
	H_3(\bA) \cong H'_3(\bA) \oplus \bC
\end{equation*}
where $\bC$ is the subspace of scalar matrices.
The projection map $H_3(\bA) \to H'_3(\bA)$, $a\mapsto a-\frac13 \tr(a)$ is a map of $\fg_\bA$-modules.

Define a product on $H'_3(\bA)$ by
\begin{equation*}
	X.Y = X\circ Y - \frac13 \tr(X\circ Y)
\end{equation*}
This is a map of $\fg_\bA$-modules, $H'_3(\bA)\otimes H'_3(\bA)\to H'_3(\bA)$.

Define $N(X)$, the \Dfn{norm} of $X\in H'(\bA)$, by $N(X) = (X,X)$.
Then the \Dfn{fundamental identity} is that, for $X\in H'_3(\bA)$.
\begin{equation*}
	N(X)^2 = N(X^2)
\end{equation*}

\subsection{Cubic forms}\label{sec:forms}
The abstract theory of cubic Jordan algebras is given in the text books
\cite{jacobson1968}, \cite{springer1997}, \cite{knus1998}, \cite[Part~II~4]{mccrimmon2004}, \cite[VI]{garibaldi2024}.

A \Dfn{Jordan algebra} is a commutative algebra that satisfies the Jordan identity.
The polarised Jordan identity is:
\begin{equation*}
	((xz)y)w+((zw)y)x+((wx)y)z=(xz)(yw)+(zw)(yx)+(wx)(yz)
\end{equation*}
%I haven't seen this written as a string diagram.

A \Dfn{cubic Jordan algebra} is a Jordan algebra, $J$, together with a linear form $T$, a quadratic form $S$
and a cubic form $N$ such that a generic element $A\in J$ satisfies
\begin{equation*}
	A^3 - T(A) A^2 + S(A)A - N(A)I
\end{equation*}
Define the adjoint by 
\begin{equation*}
	A^\# = A^2 - T(A) A + S(A)I
\end{equation*}
Then $AA^\# = N(A)A = A^\#A$.
Let $N$ be a symmetric cubic form on a finite dimensional vector space $V$. Then the first linearisation is $N(x;y)$ which is the coefficient
of $t$ in the expansion of $N(x+ty)$, so
\begin{equation*}
	N(x+ty) = N(x) + tN(x;y) + t^2N(y;x) + t^3N(y)
\end{equation*}
Note that $N(x;y)$ is quadratic in $x$ and linear in $y$. The second linearisation is 
\begin{equation*}
	N(x,y,z) = N(x+z;y) - N(x;y) - N(z;y)
\end{equation*}
which comes from the expansion
\begin{equation*}
	N(x+tz;y) = N(x;y) + tN(x,y,z) + t^2N(z;y)
\end{equation*}
Note that $N(x,y,z)$ is linear in all three variables.

The cubic form can be recovered from the polarisation (if 6 is invertible) since
\begin{equation*}
	N(x) = 6N(x,x,x)
\end{equation*}

A \Dfn{basepoint} is an element $c\in V$ with $N(c)=1$. Define linear forms
\begin{equation*}
	T(x) \coloneqq N(c;x) \quad,\quad
T(x,y) \coloneqq T(x)T(y) - N(c,x,y)
\end{equation*}
Define a quadratic form and its polarisation
\begin{equation*}
	S(x) \coloneqq N(x;c) \quad,\quad
	S(x,y) \coloneqq N(x,y,c)
\end{equation*}
These satisfy
\begin{equation*}
	T(c)=3 \quad,\quad S(c)=3 \quad,\quad T(c,y) = T(y)
\end{equation*}

Assume the bilinear trace, $T$, is a nondegenerate bilinear form. Then we define the adjoint by
\begin{equation*}
	T(x^\#,y) = N(x;y)
\end{equation*}
This satisfies $S(x) = T(x^\#)$.

The identity $N(x^\#) = N(x)^2$ appears in \cite[Problem~4.5]{mccrimmon2004}.

A \Dfn{Jordan cubic norm} is a cubic norm that satisfies
\begin{equation*}
	x^{\#\,\#} = N(x)x
\end{equation*}

We can construct a cubic Jordan algebra from a Jordan cubic norm.
The map $x\mapsto x^\#$ is quadratic. The linearisation is a bilinear product
\begin{equation*}
	x\# y \coloneqq (x+y)^\# - x^\# - y^\#
\end{equation*}
This satisfies $S(x,y) = T(x\# y)$.

The Jordan product is determined by the sharp product
\begin{equation*}
	x \circ y \coloneqq \frac12 \left(x\# y + T(x)y + T(y)x - S(x,y)c \right)
\end{equation*}

The polarisation of the fundamental identity, using the symmetry of the product and the inner product, is
\begin{equation}
	(X\circ Y, Z\circ W) + (X\circ Z, Y\circ W) + (X\circ W,Y\circ Z)
	= (a,d)(b,c) + (a,b)(c,d) + (a,c)(b,d)
\end{equation}
for $X,Y,Z,W\in H'_3(\bA)$.

The conclusion is Springer's result that a cubic Jordan algebra has a Jordan cubic norm
(by construction) and that a Jordan cubic norm together with a basepoint is a cubic Jordan algebra.
%The small print is that we are working with finite dimensional vector spaces and 6 is invertible.

\subsection{Lie algebras}
Let $N$ be a Jordan cubic norm on $V$. Define the \Dfn{structure group} to be the subgroup of $\Aut(V)$
which preserves $N$. Note that the structure group is an algebraic group. 

Let $J$ be a cubic Jordan algebra. Then we have the structure group and the derivation algebra,
$\der(J)$. Define the structure algebra, $\str(J)$, to be the Lie algebra of the structure group
of the Jordan cubic norm. By construction, $\der(J)$ is a subalgebra of $\str(J)$.

The action of $\der(J)$ on $J$ preserves the unit, $c$. Define $J'\subset J$ to be the subspace orthogonal
to $c$, equivalently, the kernel of $T$. Then $J$, as a module for $\der(J)$, decomposes as $J\cong J'\oplus I$.
Also, $\str(J)$,  as a module for $\der(J)$, decomposes as $\str(J) \cong \der(J)\oplus J'$.

Restricting to the kernel of $T$,
\begin{equation*}
	x \circ y \coloneqq \frac12 \left(x\# y - S(x,y)c \right)
\end{equation*}
and then projecting to the orthogonal complement to $c$ gives $\frac12 x\# y$.
This shows that the sharp product is a bilinear product on $J'$ and that $\der(J)$ is the Lie algebra
of derivations.

Then the quadratic norm on $J'$ satisfies
$S(x\# x) = S(x)^2$.

%-------------------------------------------------
\section{Derivation algebras}
%-------------------------------------------------
\subsection{String diagrams}
The string diagrams are discussed in \cite[Chapter~19]{Cvitanovic2008} (our convention is $\alpha=(\nn+2)$),
 \cite{gandhi2022} (our convention is $\delta=\nn$ and $\alpha=(\nn+2)$) and \cite[\S6.2.3,\S6.7]{morrison2024}.
 
 %For connected simple trivalent graphs of degree five, see \href{https://oeis.org/A014372}{Degree 3, Girth 5},
 %and \href{https://houseofgraphs.org/meta-directory/cubic}{Cubic graphs}
 %and \href{https://www.mathe2.uni-bayreuth.de/markus/reggraphs.html}{Regular Graphs}.

%The free category is the cobordism category of trivalent graphs.
In \S\ref{sec:cubic} we described the algebraic theory of a cubic Jordan algebra. In this section we
consider the algebraic theory for the trace-free subspace of a cubic Jordan algebra and describe the PROP
for this theory in terms of string diagrams. The data for this theory is a vector space with a non-degenerate
symmetric inner product and a fully symmetric cubic form $N$. This data determines a bilinear multiplication
$(x,y)\mapsto x\# y$ which satisfies 
\[ (x,y\# z) = N(x,y,z) = (x\#y,z) \]
The cobordism category of trivalent graphs is a strict rigid symmetric monoidal category
and has monoid of objects $\bN$ (with addition) and so is a PROP. The cubic norm on the trace-free subspace
of a cubic Jordan algebra satisfies the \Dfn{fundamental relation}:
\[ (x\# x,x\# x) = (x,x)^2 \]
In order to write this in terms of string diagrams we first take the polarisation.
The polarisation of the fundamental relation is:
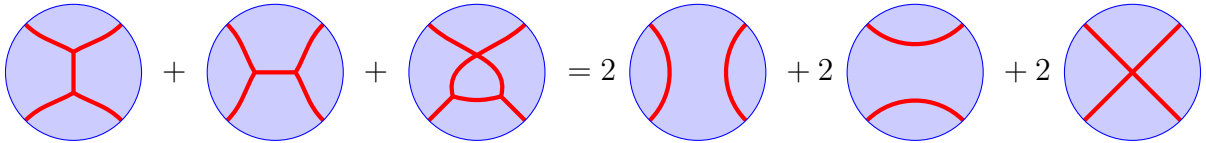
\begin{equation*}
\Kcirc + \Hcirc + \Acirc = 2\Icirc +2 \Ucirc +2 \Xcirc
\end{equation*}
The polarisation introduces coefficients.
In order to polarise a cubic norm we require 6 to be invertible.
Note that we can rescale  a trivalent vertex. The fundamental relation fixes a scale.

In addition to the fundamental relation we also impose the reduction rules:
\begin{figure}
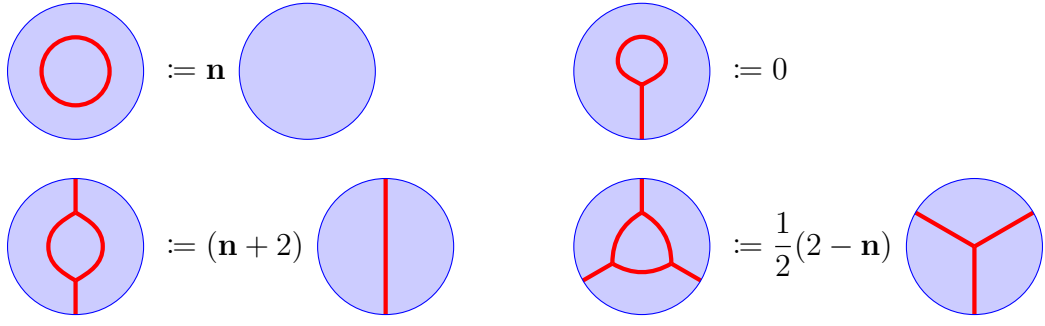

	\begin{align*}
\circle &\coloneqq \nn \empty & 	\tadpole &\coloneqq 0 \\
	\bigon &\coloneqq (\nn+2)\linecirc & 	\trianglecirc &\coloneqq \frac{1}{2} (2-\nn)\trivalent
\end{align*}
\caption{Reduction rules}\label{fig:f4reductions}
\end{figure}
The first reduction rule defines the loop value. The loop value is an element $\nn$ in the coefficient ring.
Let $\bZ[1/6,\nn]$ be the ring obtained from $\bZ$ by inverting 6 and adjoining the formal variable $\nn$.
Then the \Dfn{coefficient ring}, $\kk$, is a commutative $\bZ[1/6,\nn]$-algebra.

The second reduction rule is independent of the
fundamental relation and excludes an infinite sequence of cubic Jordan algebras, see \cite[Example~6.60]{morrison2024} for details. The third and
fourth reduction rules are consequences of the fundamental relation and the first two reduction rules;
see \cite[(19.2),(19.17)]{Cvitanovic2008} and \cite{gandhi2022} and Example~\ref{ex:f4source} below.

The last reduction rule is
\begin{multline}\label{eq:f4square}
	\squarecirc \coloneqq \frac{(6-\nn)}2\Acirc\\ + \frac{(3\nn+2)}2\Icirc + \frac{(3\nn+2)}2\Ucirc -\frac{(\nn+6)}2 \Xcirc \\
	= \frac{(\nn-6)}2\Kcirc + \frac{(\nn-6)}2\Hcirc \qquad\qquad\qquad\qquad \\ + \frac{(\nn+14)}2\Icirc + \frac{(\nn+14)}2\Ucirc + \frac{(6-3\nn)}2\Xcirc
\end{multline}
Note that the last reduction rule is invariant under rotation.

The PROP defined by these relations has a second interpretation as the category of invariant tensors
of a trace-free cubic Jordan algebra regarded as a representation of the automorphism group.

%The dimension of the adjoint representation is
%\[ \frac{3\nn(\nn-2)}{(\nn+10)} \]

\subsection{Confluence}
We distinguish the reduction rules from the fundamental relation. Each reduction rule is oriented and
substitutes one diagram by a linear combination. This defines a rewriting system on linear combinations of diagrams. A reduction step consists of making the substitution in one term. 

\begin{prop}
	The rewriting system arising from the reduction rules in Figure~\ref{fig:f4reductions} and \eqref{eq:f4square} is confluent.
\end{prop}
\begin{proof}
	This is an application of Newman's Lemma (aka Bergman's Diamond Lemma). It is sufficient to show that
	the rewriting system is terminating and locally confluent.
	
	Each term in the support of the linear combination has fewer vertices than the original term. Hence, by well-founded induction, the rewriting system is terminating.
	
	The proof that the rewriting system is locally confluent is a direct calculation.  The overlaps
	are shown in Figure~\ref{fig:overlapsf4}.
\end{proof}	
\begin{figure}[h]
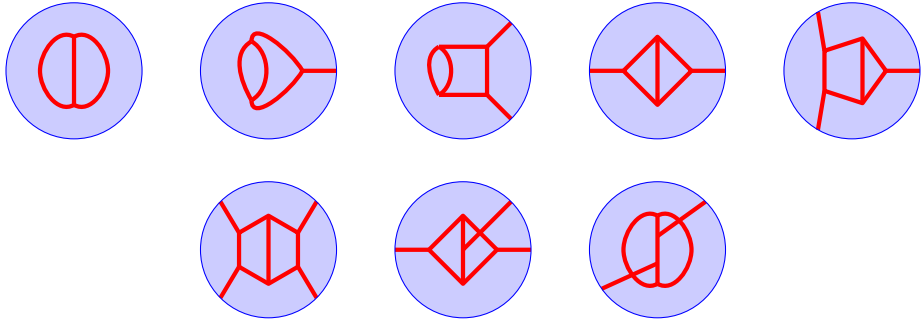

	\[ \overlaptwotwo \quad \overlaptwothree \quad \overlaptwofour %
	\quad \overlapthreethree \quad \overlapthreefour \]
	\[ \overlapfourfour \quad \overlapthreethreeA \quad \overlaptwotwoB \]
	\caption{Overlaps}\label{fig:overlapsf4}
\end{figure}

\begin{cor}\label{cor:f4} Every linear combination of string diagrams has a normal form which is a linear combination of string diagrams of girth at least five.
\end{cor}

\begin{proof} The normal form is given by applying reduction rules. Since the rewriting system is confluent
	the result is independent of the choice of the sequence of reduction steps. The result follows from the
	observation that the string diagrams which have no reduction are the string diagrams of girth at least five.
\end{proof}

\subsection{Closed graphs}

Let $\cT$ be the cobordism category of trivalent graphs. For any commutative $\bZ[1/6,\nn]$-algebra, $\kk$,
let $\kk \cT$ be the free $\kk$-linear category on $\cT$. In this section we prove:

\begin{thm}\label{thm:f4} The $\kk$-linear tensor ideal of $\kk \cT$ generated by the fundamental relation and the reduction rules
contains $p^{(F)}(\nn)[]$ where $[]$ is the empty diagram and $p^{(F)}(\nn)$ is the polynomial
\[ p^{(F)}(\nn)\coloneqq (\nn - 26) (\nn - 14) (\nn - 8) (\nn - 5)(\nn - 2)^2 \nn (\nn + 1) (\nn + 2) \]
\end{thm}

\begin{cor}
	Let $K$ be a field whose characteristic is not 2 or 3.
	The dimension of a trace-free cubic Jordan algebra over $K$ is an element of
	$\{26,14,8,5,2,0,-2\}$.
\end{cor}

\begin{proof} A trace-free cubic Jordan algebra over $K$ gives a $K$-linear symmetric monoidal functor from $K\cT$ to the category of finite dimensional super vector spaces over $K$. Put $\kk=K[\nn]$.
The homomorphism $K[\nn]\to K$ maps $\nn$ to the super dimension of the trace-free cubic Jordan algebra
and also must map $p^{(F)}(\nn)$ to 0.
\end{proof}

\begin{rem}
Consider $K[\nn]\cT$ as a family of trace-free cubic Jordan algebras over $K$ parametrised by the affine line
$\Spec(K[\nn])$. Then Theorem~\ref{thm:f4} says that this is a family parametrised by
the quotient ring $\Spec(K[\nn]/\left\langle p^{(F)}(\nn) \right\rangle)$.
\end{rem}

We evaluate trivalent graphs. Let $G$ be a trivalent graph. An \Dfn{evaluation} of $G$ is a relation of the form $G=p_G(\nn)\,[]$ where $p_G(\nn)\in \kk$ and $[]$ is the empty graph.
We refer to $p_G(\nn)$ as the evaluation. For example, the evaluation of the loop is $\nn$.
The evaluation of a trivalent graph is the product of the evaluations of the connected components, so we assume $G$ is connected.

Applying Corollary~\ref{cor:f4}, the normal form of a trivalent graph is a
linear combination of trivalent graphs of girth at least five. The method for finding evaluations
of connected trivalent graphs is inductive on the number of vertices. There are no connected trivalent
graphs with less than 10 vertices. This implies that the normal form of a trivalent graph with less
than 10 vertices is its evaluation.

A \Dfn{source} is a connected graph with one vertex of valence four and all other vertices of valence three.
Given a source we can substitute the six-term relation for the four valent vertex to get a relation which is a linear combination of six trivalent graphs. Note that this is well-defined as the six-term relation is invariant under permutation of the boundary points, by construction.

\begin{ex}\label{ex:f4source} The third and fourth reduction rules are obtained from the
	sources
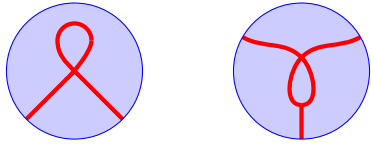
\begin{equation*}
\curl \qquad \twist
\end{equation*} 
\end{ex}

If the source has $t$ trivalent vertices then three of the terms have $t$ vertices and three have $t+2$ vertices. The normal form of this relation is a linear combination
of connected trivalent graphs of girth at least five. Furthermore, each term of the relation has at most
$t+2$ vertices.

This gives an inductive procedure for finding evaluations. Assume we have found evaluations of all
trivalent graphs of girth at least five with fewer than $t$ vertices. Given a source with $t-2$
trivalent vertices we construct the relation and its normal form. The normal form is a linear combination
of trivalent graphs where each term has girth at least five and at most $t$ vertices. If any term has a
connected component with fewer than $t$ vertices then, by the inductive assumption, we have found the
evaluation and we substitute the evaluation for the component. The result of these substitutions is
a linear combination of connected trivalent graphs of girth at least five and precisely $t$ vertices.

\begin{ex}
The Petersen graph is the unique (up to isomorphism) connected trivalent graph with 10 vertices.
There is also a unique (up to isomorphism) source with 8 trivalent vertices. Substituting the 
six term relation gives a linear combination of six connected trivalent graphs. One of the terms is
the Petersen graph and the other five terms have girth at most four. Therefore the normal form of this
relation is the evaluation of the Petersen graph.
\end{ex}

The number of trivalent graphs of girth at least five and the number of sources is shown in Table~\ref{table:f4}. For more terms in the first sequence see \oeis{A014372}.

\begin{table}[h]
\begin{equation*}
\begin{array}{r|rrrr}
& 10 & 12 & 14 & 16 \\ \hline
graphs & 1 & 2 & 9 & 49 \\
sources & 1 & 4 & 31 & 335
\end{array}
\end{equation*}
\caption{Trivalent graphs and sources}\label{table:f4}
\end{table}

We put $\kk=\bQ[\nn]$ and, for $t=10,12,14,1$ we generated all sources
with $t-2$ trivalent vertices (up to isomorphism) and computed the linear combination of
connected trivalent graphs of girth at least five and $t$ vertices. We observed that for each 
connected trivalent graph there was a source that gave an evaluation. For $t<16$ substituting these
evaluations in the remaining relations gave 0. However for $t=16$ 145 of the 335 sources gave a non-zero polynomial. For all sources which gave a non-zero polynomial the polynomial was of the form
\[ \pm \frac{3^r}{2^s} (\nn - 26) (\nn - 14) (\nn - 8) (\nn - 5)(\nn - 2)^2 \nn (\nn + 1) (\nn + 2) \]
where $r,s$ are positive integers. These are all associates of $p^{(F)}(\nn)$ as elements of $\bZ[1/6,\nn]$.

%\cite{comes2017}

Following \cite{Thurston2004}, we construct a trace-free Jordan algebra for each root of $p^{(F)}(\nn)$.
This shows that none of the roots of $p^{(F)}(\nn)$ will be eliminated by continuing the calculation
to $t>16$. 
\begin{itemize}
	\item The trace-free cubic Jordan algebras
	for $\nn = 26,14,8,5$ are the algebras $H_3'(\bA)$ for $\bA$ a composition algebra constructed in \S\ref{sec:cubic}.
	\item The trace-free cubic Jordan algebra for $\nn=2$ is the two dimensional irreducible representation
	of the symmetric group $\fS_3$.
	\item The case $\nn=0$ is the degenerate case where all diagrams are zero.
	\item The trace-free cubic Jordan algebra for $\nn=-1$ is the super space of dimension $(k|k+1)$ and the automorphism
	group is the orthosymplectic super group $\OSp(1|2)$. The PROP reduces to a specialisation of the
	Brauer category.
	\item The trace-free cubic Jordan algebra for $\nn=-2$ is the super space of dimension $(0|2)$ and the automorphism group is $\SU(2)$. The PROP reduces to the Temperley-Lieb category.
\end{itemize}

The six diagrams in the fundamental relation generically span a free $\kk$-module of rank 5.
This $\kk$-module has a symmetric inner product, defined diagrammatically. This inner product is degenerate
for $\nn= 2,0,-1,-2$ and is non-degenerate otherwise. For $\nn= 2,0,-1,-2$ there are further relations
give by the null space of the inner product.

The trace-free Jordan algebra for $\nn=2$ is constructed in \cite{Deligne2002}.
Let $\mu_2$ be the group of order 2. 
The centraliser of $G_2\subset \Spin(8)$ is $\mu_2^2$, the centre of $\Spin(8)$.
Let $\fso(8)$ be the adjoint representation of $\Spin(8)$ and $\fg_2$ the adjoint representation of $G_2$.
The automorphism group of $\fso(8)$ is $\PSO(8) \rtimes \fS_3$.
The centraliser of $G_2\subset \PSO(8) \rtimes \fS_3$ is the symmetric group $\fS_3$.

There is an isomorphism of $G_2\times \fS_3$-modules $\fso(8) \cong \fg_2\otimes 1 \oplus V\otimes W$
where $V$ is the seven dimensional fundamental representation of $G_2$ and $W$ is the two dimensional
irreducible representation. The vector space $W$ is a trace-free cubic Jordan algebra whose automorphism group is $\fS_3$.

%-------------------------------------------------
\section{Structure algebras}
%-------------------------------------------------
\subsection{String diagrams}
The string diagrams are discussed in \cite[\S2.3,Chapter~18]{Cvitanovic2008}
(our convention is $\alpha=2(\nn+3)$).

In \S\ref{sec:cubic} we described the algebraic theory of a cubic norm. In this section we
consider the algebraic theory for the cubic norm of a cubic Jordan algebra and describe the PROP
for this theory in terms of string diagrams. The data for this theory is a vector space, $V$, together
with a  fully symmetric cubic form $N$. This data determines a linear map $V\otimes V\to V^\ast$
(where $V^\ast$ is the dual vector space):
$(x,y)\mapsto x\# y$ which satisfies 
\[ (x,y\# z) = N(x,y,z) = (x\#y,z) \]
The diagram category, $\cB$, is the cobordism category of oriented trivalent graphs such that every trivalent
vertex is either a source or a sink, see Figure~\ref{fig:sink}. The category $\cB$ is a rigid symmetric
category and the monoid of objects is the free monoid on two generators. This makes $\cB$ a coloured PROP
with two colours. The two generating objects are a dual pair. Note that any cycle has even length.

\begin{figure}[h]
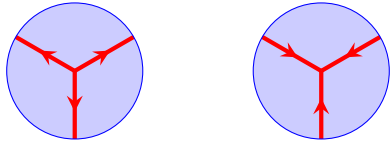

	\source \qquad \sink
	\caption{Source and sink}\label{fig:sink}
\end{figure}

The polarisation of the fundamental relation is the seven-term relation
shown in Figure~\ref{fig:seventerm}
\begin{figure}[h]
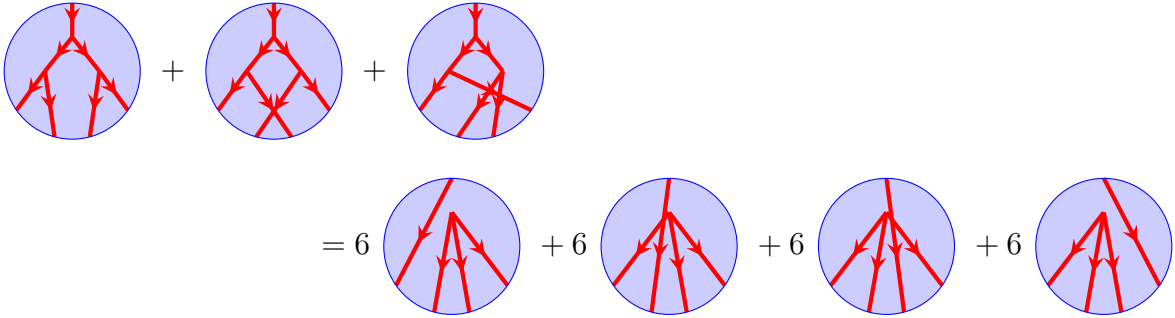

	\begin{multline*}
		\Wa + \Wb + \Wc \\ = 6 \Vu + 6 \Vv + 6 \Vw + 6 \Vx
	\end{multline*}
	\caption{Seven term relation}\label{fig:seventerm}
\end{figure}

The polarisation introduces coefficients.
In order to polarise a cubic norm we require 6 to be invertible.
Note that we can rescale  a trivalent vertex. The fundamental relation fixes a scale.

In addition to the fundamental relation we also impose the reduction rules:
\begin{figure}
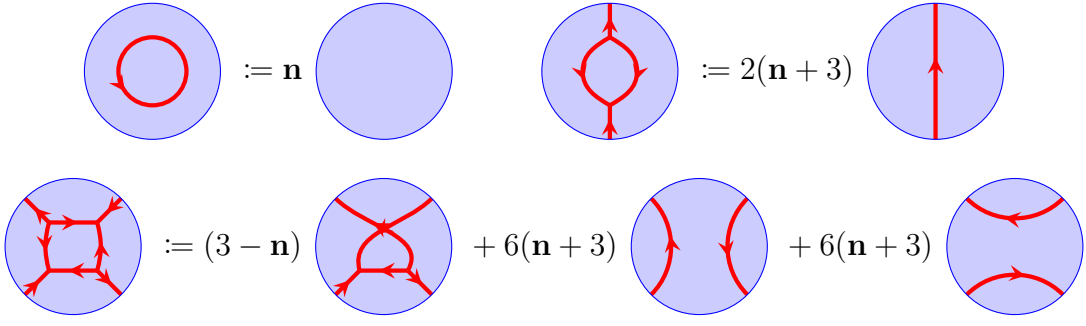

\begin{gather*}
	\circledir \coloneqq \nn \empty
	\qquad
	\bigondir \coloneqq 2(\nn+3)\linedir \\
	\squaredir \coloneqq (3-\nn) \Adir + 6(\nn+3) \Idir + 6(\nn+3) \Udir 
\end{gather*}
	\caption{Reduction rules}\label{fig:e6reductions}
\end{figure}
The first reduction rule defines the loop value. The loop value is an element $\nn$ in the coefficient ring.
Let $\bZ[1/6,\nn]$ be the ring obtained from $\bZ$ by inverting 6 and adjoining the formal variable $\nn$.
Then the \Dfn{coefficient ring}, $\kk$, is a commutative $\bZ[1/6,\nn]$-algebra.

The coloured PROP defined by these relations has a second interpretation as the category of invariant tensors
of a cubic norm regarded as a representation of the structure group.

\subsection{Confluence}
We distinguish the reduction rules from the fundamental relation. Each reduction rule is oriented and
substitutes one diagram by a linear combination. This defines a rewriting system on linear combinations of diagrams. A reduction step consists of making the substitution in one term. 

\begin{prop}
	The rewriting system arising from the reduction rules in Figure~\ref{fig:e6reductions} is confluent.
\end{prop}
\begin{proof}
	This is an application of Newman's Lemma (aka Bergman's Diamond Lemma). It is sufficient to show that
	the rewriting system is terminating and locally confluent.
	
	Each term in the support of the linear combination has fewer vertices than the original term. Hence, by well-founded induction, the rewriting system is terminating.
	
	The proof that the rewriting system is locally confluent is a direct calculation.  The overlaps
	are shown in Figure~\ref{fig:overlapse6}.
\end{proof}

\begin{figure}[h]
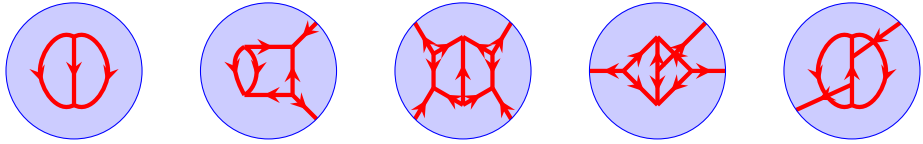

	\[ \overlaptwotwodir \quad \overlaptwofourdir \quad 
	\overlapfourfourdir \quad \overlapthreethreeAdir \quad \overlaptwotwoBdir \]
	\caption{Overlaps}\label{fig:overlapse6}
\end{figure}

\begin{cor}\label{cor:e6} Every linear combination of string diagrams has a normal form which is a linear combination of string diagrams of girth at least six.
\end{cor}

\begin{proof} The normal form is given by applying reduction rules. Since the rewriting system is confluent
	the result is independent of the choice of the sequence of reduction steps. The result follows from the
	observation that the string diagrams which have no reduction are the string diagrams of girth at least six.
\end{proof}

\subsection{Closed graphs}

Let $\cB$ be the cobordism category of oriented trivalent graphs. For any commutative $\bZ[1/6,\nn]$-algebra, $\kk$,
let $\kk \cB$ be the free $\kk$-linear category on $\cB$. In this section we prove:

\begin{thm}\label{thm:e6} The $\kk$-linear tensor ideal of $\kk \cB$ generated by the fundamental relation in Figure~\ref{fig:seventerm} and the reduction rules in Figure~\ref{fig:e6reductions}
contains $p^{(E)}(\nn)[]$ where $[]$ is the empty diagram and $p^{(E)}(\nn)$ is the monic polynomial
\[ p^{(E)}(\nn) \coloneqq (\nn - 27) (\nn - 15) (\nn - 9) (\nn - 6)(\nn - 3)^2(\nn - 1)\nn^2 (\nn + 1) (\nn + 3)^2 \]
\end{thm}

\begin{cor}
Let $K$ be a field whose characteristic is not 2 or 3.
The dimension of a cubic norm over $K$ is an element of
$\{27,15,9,6,3,1,-1,-3\}$.
\end{cor}

\begin{proof} A cubic norm over $K$ gives a $K$-linear symmetric monoidal functor from
$K\cB$ to the category of finite dimensional super vector spaces over $K$. Put $\kk=K[\nn]$.
The homomorphism $K[\nn]\to K$ maps $\nn$ to the super dimension of the cubic norm
and also must map $p^{(E)}(\nn)$ to 0.
\end{proof}

\begin{rem}
	Consider $K[\nn]\cB$ as a family of cubic norms over $K$ parametrised by the affine line
$\Spec(K[\nn])$. Then Theorem~\ref{thm:e6} says that this is a family parametrised by
the quotient ring $\Spec(K[\nn]/\left\langle p^{(E)}(\nn) \right\rangle)$.
\end{rem}

We evaluate oriented trivalent graphs. Let $G$ be an oriented trivalent graph. An \Dfn{evaluation} of $G$ is a relation of the form $G=p_G(\nn)\,[]$ where $p_G(\nn)\in \kk$ and $[]$ is the empty graph.
We refer to $p_G(\nn)$ as the evaluation. For example, the evaluation of the loop is $\nn$.
The evaluation of a trivalent graph is the product of the evaluations of the connected components, so we assume $G$ is connected.

Applying Corollary~\ref{cor:e6}, the normal form of an oriented trivalent graph is a
linear combination of trivalent graphs of girth at least six. The method for finding evaluations
of connected trivalent graphs is inductive on the number of vertices. There are no connected trivalent
graphs with less than 14 vertices. This implies that the normal form of a trivalent graph with less
than 14 vertices is its evaluation.

A \Dfn{oriented source} is a connected oriented graph with one vertex of valence two connected to a vertex of valence four and all other vertices of valence three.
Given an oriented source we can substitute the seven-term relation for the vertex of valence two connected to the vertex of valence four to get a relation which is a linear combination of seven trivalent oriented graphs. Note that this is well-defined as the seven term relation is invariant under permutation of four of the five boundary points, by construction.
If the source has $t$ trivalent vertices then four of the terms have $t$ vertices and three have $t+2$ vertices. The normal form of this relation is a linear combination
of connected trivalent oriented graphs of girth at least six. Furthermore, each term of the relation has at most
$t+2$ vertices.

This gives an inductive procedure for finding evaluations. Assume we have found evaluations of all
trivalent graphs of girth at least six with fewer than $t$ vertices. Given an oriented source with $t-2$
trivalent vertices we construct the relation and its normal form. The normal form is a linear combination
of trivalent oriented graphs where each term has girth at least six and at most $t$ vertices. If any term has a
connected component with fewer than $t$ vertices then, by the inductive assumption, we have found the
evaluation and we substitute the evaluation for the component. The result of these substitutions is
a linear combination of connected trivalent oriented graphs of girth at least six and precisely $t$ vertices.

\begin{ex}
	The Heawood graph is the unique (up to isomorphism) connected trivalent bipartite graph with 14 vertices.
	There is also a unique (up to isomorphism) source with 11 trivalent vertices. Substituting the 
	seven term relation gives a linear combination of seven connected trivalent bipartite graphs. One of the terms is
	the Heawood graph and the other six terms have girth at most four. Therefore the normal form of this
	relation is the evaluation of the Heawood graph.
\end{ex}

The number of trivalent oriented graphs and the number of oriented sources is shown in Table~\ref{table:e6}.
For more terms in the first sequence see \oeis{A260811}.
\begin{table}[h]
	\begin{equation*}
		\begin{array}{r|rrrrr}
			& 14 & 16 & 18 & 20 & 22 \\ \hline
			graphs & 1 & 1 & 3 & 10 & 28 \\
			sources & 1 & 1 & 6 & 47 & 406
		\end{array}
	\end{equation*}
	\caption{Trivalent oriented graphs and oriented sources}\label{table:e6}
\end{table}

We put $\kk=\bQ[\nn]$ and, for $t=14,16,18,20,22$ we generated all oriented sources
with $t-2$ trivalent vertices (up to isomorphism) and computed the linear combination of
connected trivalent oriented graphs of girth at least six and $t$ vertices. We observed that for each 
connected trivalent oriented graph there was an oriented source that gave an evaluation. For $t<22$ substituting these
evaluations in the remaining relations gave 0. However for $t=22$ 149 of the 406 sources gave a non-zero polynomial. For all oriented sources which gave a non-zero polynomial the polynomial is an associate of $p^{(E)}(\nn)$ as elements of $\bZ[1/6,\nn]$.

There is a strict rigid symmetric monoidal functor from the Structure PROP to the Derivation PROP.
This functor is $\bQ[\nn]$-linear after the substitution $\nn_E = \nn_F+1$. This predicts that
$p_F(\nn+1)$ divides $p_E(\nn)$. A direct comparison of the two polynomials gives
\[ p_E(\nn) = p_F(\nn+1)\nn(\nn+3)^2 \]
This constructs a cubic norm for all solutions of $p_E(\nn)$ except $\nn=0$ and $\nn=-3$. The functor
for $\nn=0$ sends all diagrams to 0.

We explain the factor $(\nn+3)^2$ by constructing a pair of $\kk$-linear strict rigid monoidal functors
from the PROP to the category of finitely generated free super $\kk$-modules.
Let $V$ be a free $\kk$-module of rank 3 considered as an odd super module. Since $V$ is odd, the
determinant identifies $S^3(V) \cong \kk$ and the vector cross product is an identification $S^2(V)\cong V^\ast$.
The structure group is $\SL(3,\kk)$. Since $V$ and $V^\ast$ are inequivalent representations of 
$\SL(3,\kk)$ this constructs a pair of functors.

\bibliographystyle{alphaurl}
\bibliography{F4series}

%-------------------------------------------------
\end{document}